\documentclass[12pt]{amsart}
\usepackage{pstricks}
\usepackage{fullpage}
\usepackage{amsfonts}
\usepackage{amssymb}
\usepackage{amsthm}
\usepackage{latexsym,amsmath}
\usepackage{graphicx}

\newtheorem{theorem}{Theorem}[section]
\theoremstyle{definition}
\newtheorem{Definition}[theorem]{Definition}

\newenvironment{theorem*}[1]{\medskip
                            \noindent
                            {\bf Theorem #1. }\ %
                            \begingroup \sl}
                            {\endgroup\medskip}

\newtheorem{lemma}[theorem]{Lemma}

\newcommand\Prefix[3]{\vphantom{#3}#1#2#3}

\DeclareMathOperator{\Cham}{Cham}

\title{A Local-To-Global Result for Topological Spherical Buildings}

\author[$\mathrm{M^{\lowercase{c}}Callum}$]{Rupert $\mathrm{M^{\lowercase{c}}Callum}$}

\begin{document}

\address{%
University of M\"unster, 
Einsteinstrasse 62, 
48149 M\"unster, 
Germany}

\email{rupertmccallum@yahoo.com}

\date{January 30, 2012}

\maketitle

\begin{abstract}

Suppose that $\Delta, \Delta'$ are two buildings each arising from a semisimple algebraic group over a field, a topological field in the former case, and that for both the buildings the Coxeter diagram has no isolated nodes. We give conditions under which a partially defined injective chamber map, whose domain is the subcomplex of $\Delta$ generated by a nonempty open set of chambers, and whose codomain is $\Delta'$, is guaranteed to extend to a unique injective chamber map. Related to this result is a local version of the Borel-Tits theorem on abstract homomorphisms of simple algebraic groups. 

\end{abstract}

\bigskip

Mathematicsl Subject Classification 2000: 51H10

\section{Introduction}

 Throughout the history of Lie theory there has been a notion of ``local isomorphism". Indeed, the original notion of Lie groups considered by
Sophus Lie in \cite{LieEngel1888} was an essentially local one. Suppose that $G$ is an algebraic group defined over a Hausdorff topological field $k$. One may consider the notion of a local $k$-isogeny from the group $G$ to a group $G'$ defined over $k$ of the same dimension. This is a mapping defined on a nonempty open neighbourhood of the identity of $G(k)$ in the strong $k$-topology, (see Definition \ref{strongtopology}), which ``locally" acts as a $k$-isogeny, in the sense that it is a local $k$-homomorphism and its range is Zariski dense. As far as I know this notion has not been investigated systematically. In this paper we shall produce a local version of the Borel-Tits theorem which hints at the possibility that it may be fruitful to investigate this notion.

\bigskip

The Borel-Tits result was used by Mostow \cite{Mostow73} to prove his ``strong rigidity theorem", and also underpins the work of Margulis on ``superrigidity" \cite{Margulis91}. More recently, using Tanaka's theory of prolongations of maps of filtered structures on manifolds \cite{Tanaka70}, Yamaguchi \cite{Yamaguchi93} showed that smooth local maps preserving the fibrations in a Tits building in the real case must arise form the the action of the associated semisimple group. This suggests that there might be a ``local rigidity theorem" for buildings which does not require any assumptions of smoothness, or even continuity, and in this paper we give a local-to-global result for the buildings of semisimple algebraic groups which establishes this. We now proceed to the description of the results.

\bigskip

There are two famous theorems of Jacques Tits which are discussed in \cite{Tits74} which say, roughly, that a thick spherical building is completely determined by a very small portion of it. The first theorem is as follows.

\begin{Definition} Suppose that $\Delta$ is a building and $C$ is a chamber of $\Delta$, and that $n$ is a natural number. We define $E_{n}(C):=\{D\in$ Cham $\Delta\mid$ codim $(C\cap D)\leq n\}$. \end{Definition}

\begin{theorem}[Rigidity theorem] \label{rigiditytheorem} Let $\Delta$ be a thick spherical building, and let $C, C'$ be opposite chambers of $\Delta$. If an automorphism $\varphi$ of $\Delta$ fixes $E_{1}(C)\cup\{C'\}$ pointwise, then $\varphi$ is the identity. \end{theorem}

The next theorem gives circumstances under which a bijection from a small part of a building to a small part of another building extends to a building isomorphism.

\begin{theorem}[Extension theorem] Let $\Delta, \Delta'$ be thick irreducible spherical buildings of rank at least 3. Let $\varphi:E_{2}(C)\rightarrow E_{2}(C')$ be an adjacency-preserving bijection for some $C\in\Delta$ and $C'\in\Delta'$. Then $\varphi$ extends to an isomorphism $\Delta\rightarrow\Delta'$. \end{theorem}

This paper proves a similar kind of theorem for the topological spherical building associated with a semisimple algebraic group $G$ over a non-discrete Hausdorff topological field $k$, provided that the Coxeter diagram of the building has no isolated nodes. First we must give the definition of topological spherical building we are using. Our definition is taken from \cite{Kramer02}. Another definition of topological spherical building is given in \cite{Burns87} but it is now known that this is not a good definition to use in the non-compact case, because, for example, given a topology that satisfies the definition one can easily enlarge the topology in an arbitrary way and the definition is still satisfied.

\begin{Definition} Suppose that $\Delta$ is a spherical building of type $(W,S)$, considered as a simplicial complex. Suppose that for each $s\in S$ we are given a Hausdorff topology on the set of vertices of type $s$. We thereby obtain, for each $J\subseteq S$, a Hausdorff topology on the simplices of type $J$, by viewing the simplices of type $J$ as $J$-tuples of vertices and taking the subspace topology arising from the product topology. Given any two subsets $J, K \subseteq S$, consider the set $\mathfrak{D}^{J,K}:=\{(X,Y)\mid X$ has type $J,$ $Y$ has type $K,$ and there exist opposite chambers containing $X$ and $Y\}$ and the mapping $\mathfrak{D}^{J,K}\rightarrow$ Cham $\Delta, (X,Y)\mapsto\mathrm{proj}_{X} Y$. Suppose that all of these mappings are continuous. Then the building is said to be a topological spherical building. \end{Definition}

We will prove in Section 2 that if $G$ is a semisimple algebraic group over a Hausdorff topological field $k$, then the building $\Delta:=\Delta(G,k)$ is always a topological spherical building.

\bigskip

Now, suppose that $\Delta$ is a topological spherical building arising from a semisimple algebraic group in this way, with Coxeter system $(W,S)$, and suppose that the Coxeter diagram of $\Delta$ has no isolated nodes, so that the building is strictly Moufang in the sense of \cite{Abramenko08}. We will now define a certain base for the topology on the set of chambers.

\begin{Definition} \label{basicopen} Suppose that $C, C'$ are two opposite chambers in $\Delta$. For each $s\in S$ let $U_{s}$ be the root group corresponding to the root of the apartment containing $C$ and $C'$ which does not contain $C$ but which is attached to $C$ along a boundary panel of cotype $\{s\}$. There is a topology on $U_{s}$ arising from the topology on the set of all chambers which are attached to the root along this panel. For each $s\in S$, let $N_{s}$ be an open neighbourhood of the identity in $U_{s}$ with respect to this topology. Let $w_{0}$ be the longest word in the Coxeter group $(W,S)$, with a fixed reduced decomposition $s_{1}s_{2}\ldots s_{n}$. Let $U$ be the set of all chambers which arise from $C$ under the action of the set $N_{s_{1}}N_{s_{2}}\ldots N_{s_{n}}$. Then $U$ is said to be a basic open set, or a basic open set with respect to the pair of opposite chambers $(C,C')$. The empty set is also said to be a basic open set. \end{Definition}

We will prove in Section 2 that every basic open set is indeed an open set and that in this way a base for the topology on the set of chambers is defined.

\begin{Definition} \label{quasiconnected} Suppose that $U$ is an open set of chambers in $\Delta$. Suppose that, given any two chambers $C, D \in U$, there exists a sequence $(U_{1}, U_{2}, \ldots U_{r})$ of basic open subsets of $U$, such that for each integer $i$ such that $1\leq i<r$ we have $U_{i}\cap U_{i+1}\neq\emptyset$, and $C\in U_{1}, D\in U_{r}$. Then $U$ is said to be quasi-connected. \end{Definition}

In the case where the field $k$ from which the building $\Delta$ arises is equal to $\mathbb{R}$ or $\mathbb{C}$, it can be shown that every open connected set is quasi-connected.

\bigskip

We can now state our result. 

\begin{theorem} [The Main Theorem] \label{maintheorem} Suppose that $G$ (resp. $G'$) is a semisimple algebraic group defined over a field $k$ (resp. $k'$). Suppose that the $k$-rank of $G$ is equal to the $k'$-rank of $G'$. Suppose that $k$ is a non-discrete Hausdorff topological field. Let $\Delta:=\Delta(G,k), \Delta':=\Delta(G',k')$. Suppose that the Coxeter diagrams of $\Delta, \Delta'$ have no isolated nodes. Suppose that $U$ is a nonempty open quasi-connected subset of $\mathrm{Cham}$ $\Delta$ and $\Delta(U)$ is the subcomplex of $\Delta$ consisting of all faces of members of $U$. Then an injective chamber map $\Delta(U)\rightarrow\Delta'$ has a unique extension to an injective chamber map $\Delta\rightarrow\Delta'$. \end{theorem}

\bigskip

This theorem generalises results given in \cite{McCallum09}.\footnote{It should be noted that Chapter 6 of \cite{McCallum09} used a different definition of ``quasi-connected", and the theorems of Chapter 6 of \cite{McCallum09} in which that term appears are actually not true as they stand. For example, if we consider the disjoint union of the interiors of two squares separated by a distance of one unit, and if the slopes of the line segments are allowed to range from -1 to 1, then it is easy to see that a mapping defined on that domain which preserved line segments with slope in that range could be a disjoint union of two distinct projective transformations. But if some requirement were made on the set of slopes that it should match the set of points in appropriate way, which would ensure that the corresponding set of chambers in the building is gallery connected, as is clearly the case with our ``quasi-connected" sets defined here, then the theorems would come out true.} Some of the geometries considered in \cite{McCallum09} are unitals rather than buildings. It would appear that the result generalises to this context as well and we hope to explore that further in future work. In the case of the general linear group, a version of the theorem requiring smoothness of the mapping has appeared in \cite{Bertini24}, and a version requiring continuity in the real and complex case has appeared in \cite{Shiffman95}. A version which does not require continuity appears in \cite{Cap07}.

\bigskip

Related to Theorem \ref{maintheorem}, a result about topological buildings, is a group-theoretic result; a local version of the Borel-Tits theorem on abstract homomorphisms of simple algebraic groups, which appears in \cite{Tits73}. In that paper Borel and Tits prove the following result. 

\bigskip

First we give the definition of a special isogeny. 

\begin{Definition} If $G, G'$ are two absolutely almost simple algebraic groups defined over a field $k$ and $\varphi:G\rightarrow G'$ is an isogeny, and $T$ is a maximal torus of $G$ and $T'=\varphi(T)$, then the restriction of $\varphi$ to $T$ induces a homomorphism $\varphi^{*}:X^{*}(T')\rightarrow X^{*}(T)$ of the character group of $T'$ into that of $T$, and the isogeny $\varphi$ is said to be special if $\varphi^{*}$ maps the short roots of $G'$ onto roots of $G$. \end{Definition}

Now we give the statement of the Borel-Tits theorem.

\begin{theorem}[Borel-Tits theorem] \label{BorelTits} Suppose that $G$ (resp. $G'$) is a connected affine algebraic group over a field $k$ (resp. $k'$), and that $G$ is absolutely almost simple and $G'$ is absolutely simple and adjoint.  Denote by $G^{+}$ the subgroup of $G(k)$ generated by all $k$-rational points of unipotent radicals of $k$-parabolic subgroups. Then, given an abstract group homomorphism $\alpha:G^{+}\rightarrow G'(k')$ whose range is Zariski dense in $G'$, there exists a field homomorphism $\varphi:k\rightarrow k'$ and a special isogeny $\beta:\Prefix^{\varphi}{G}\rightarrow G'$ such that $\alpha=\beta\circ\varphi^{\circ}\mid_{G^{+}}$. \end{theorem}

\bigskip

Now we shall describe the local version of this theorem. First of all we need to give a definition of the strong $k$-topology on the set of $k$-rational points of a variety over a Hausdorff topological field $k$. This is discussed in the algebraically closed case in Chapter I $\S10$ of \cite{Mumford99}, where the main properties are left as an exercise for the reader. It is also discussed in the case $k=\mathbb{C}$ in Appendix III of \cite{Weil46}.

\begin{Definition} \label{strongtopology} Suppose that $V$ is an affine variety over a Hausdorff topological field $k$, and suppose that we are given a $k$-embedding of $V$ in an affine space $\mathbb{A}^{n}(\overline{k})$, where $\overline{k}$ is an algebraic closure of $k$, where the affine space is given the obvious $k$-structure. Then the topology on $k$ induces a topology on the set of $k$-rational points $\mathbb{A}^{n}(k)$ of the affine space $\mathbb{A}^{n}(\overline{k})$, and thereby induces a topology on $V(k)$, the set of $k$-rational points of the variety $V$. This topology on $V(k)$ does not depend on the choice of embedding in an ambient affine space, and is called the strong $k$-topology on $V(k)$. The definition can be generalised to non-affine varieties in the obvious way. \end{Definition}

Suppose that $G, G'$ and $G^{+}$ are as in the statement of the Borel-Tits theorem. We can give $G^{+}$ a topology arising from the strong $k$-topology on $G(k)$. We need to define a base for this topology. Note that this definition is similar to Definition \ref{basicopen} except in the context of groups rather than buildings.

\begin{Definition} Let $S$ be a maximal $k$-split torus in $G$. Let $\Phi_{k}(G,S)$ be the set of $k$-roots of $G$ with respect to $S$ and let $\Psi=\{\alpha\in\Phi_{k}(G,S)\mid\frac{1}{2}\alpha\notin\Phi_{k}(G,S)\}$. For each $\alpha\in\Psi$ let $U_{\alpha}$ be the subgroup of $G^{+}$ consisting of all $k$-rational points of the unipotent radical of the $k$-parabolic subgroup of $G$ containing $S$ corresponding to $\alpha$, and let $N_{\alpha}$ be an open neighbourhood of the identity in $U_{\alpha}$ in the strong $k$-topology on $U_{\alpha}$. As observed in \cite{Tits65b}, given a system of simple roots $\Pi\subseteq\Psi$ there is a natural one-to-one correspondence between conjugacy classes of $k$-parabolic subgroups and subsets of $\Pi$. Select a system of simple roots $\Pi$ and identify it with the set of generators $S$ for the Coxeter system $(W,S)$ of the building of $G$ over $k$. Pick a longest word $w_{0}\in W$ and let $s_{1}s_{2}\ldots s_{n}$ be a reduced decomposition. For each $i$ such that $1\leq i\leq n$, let $\alpha_{i}$ be the element of $\Pi$ corresponding to $s_{i}$. We say that $N_{-\alpha_{1}}N_{-\alpha_{2}}\ldots N_{-\alpha_{n}}N_{\alpha_{1}}N_{\alpha_{2}}\ldots N_{\alpha_{n}}$ and any right coset thereof is a basic open subset of $G^{+}$. We also say that the empty set is a basic open subset with respect to the torus $S$. \end{Definition}

We will prove in Section 2 that every basic open subset is indeed open and that a base for the strong $k$-topology on $G^{+}$ is defined in this way.

\begin{theorem} [The local Borel-Tits theorem] \label{localBorelTits} Suppose that $G, G',$ and $G^{+}$ are as in the statement of the Borel-Tits theorem and that $k$ is a non-discrete Hausdorff topological field, and give $G(k)$ the strong $k$-topology. Then a local abstract group homomorphism $\alpha:U\subseteq G^{+}\rightarrow G'(k')$, where $U$ is an nonempty basic open neighbourhood of the identity in $G^{+}$, whose range is Zariski dense in $G'$, extends to a global group homomorphism. \end{theorem}

\bigskip

A result related to a special case of this is already known for simply connected almost simple complex Lie groups. There is a classical theorem which says that, given a simply connected Hausdorff topological group $G$ and an open symmetric neighbourhood $U$ of the identity, any local homomorphism from this neighbourhood into an abstract group $H$ extends to a global group homomorphism (see Corollary A2.26 of \cite{Hofmann98}). A special case of the local Borel-Tits result follows from this.

\bigskip

In Section 2 we review the basic definitions of concepts from the theory of buildings and topological spherical buildings, and prove that the building of a semisimple algebraic group $G$ over a Hausdorff topological field $k$ is always a topological spherical building. We also present once again the definition of the notion of a ``quasi-connected" set, and the notion of a basic open subset in the context of both buildings and groups, and prove some of the basic properties of these notions. In Section 3 we present the proof of the local-to-global result for buildings and the local version of the Borel-Tits theorem. In Section 4 we examine a corollary dealing with nilpotent Lie groups which is a version of a result due to Yamaguchi not requiring the hypothesis of continuity.

\bigskip

This paper was written partly while I was employed as Research Intensive Academic at the Australian Catholic University,  partly while I was an Adjunct Lecturer at the University of New South Wales, and partly while I held a post-doctoral position at the University of M\"unster. I am grateful to these institutions for their
support. I would like to thank Bill Franzsen and Norman Wildberger for proofreading preliminary drafts, and Michael Cowling for giving helpful suggestions with the
Introduction. I am particularly grateful to Linus Kramer for taking an interest in my work and providing me with much helpful guidance about the theory of topological
buildings.

\section{Basic Definitions}

In this section we review the definitions of the basic concepts of the theory of buildings. The main references are \cite{Tits74} and \cite{Abramenko08}. For convenience we also review the definition of topological spherical building given in \cite{Kramer02}, as we did in Section 1. We will give a proof that the building of a semisimple algebraic group $G$ over a Hausdorff topological field $k$ is always a topological spherical building. We also repeat the definition of the term ``quasi-connected" and prove some of its basic properties.

\begin{Definition} A complex is a partially ordered set $(P,\leq )$ such that \newline
 \begin{enumerate}
 \item for all $v\in P$, the set $\{w\in P\mid w\leq v\}$ is order isomorphic to $(\mathfrak{P}(S),\subseteq )$ for some set $S$, $\mathfrak{P}(S)$ being the powerset of $S$. In this paper the set $S$ will always be finite. Such a partially ordered set is called a simplex.
 \item if $A, B \in P$ then $A$ and $B$ have a greatest lower bound denoted by $A\cap B$.
\end{enumerate}

\bigskip 

 If $A\leq B$ we say that $A$ is contained in $B$ or is a face of $B$.  \end{Definition}

A complex has just one minimal element called $0$. The elements which are minimal nonzero elements are called vertices and the number of vertices contained in a given element of a complex is called its rank. The rank of the complex is the supremum of the ranks of all the elements. As remarked earlier, in this paper we will only be considering complexes of finite rank. Given an element $A$ of a complex $(\Delta,\leq)$, the set $\{B\in\Delta\mid A\leq B\}$, with the order relation induced from $\Delta$, is called St $A$ or the star of $A$. It is also a complex. Given $B\in \mathrm{St}$ $A$, the rank of $B$ in St $A$ is called the codimension of $A$ in $B$. The greatest lower bound of two elements $C, D$ in a complex (which always exists) is denoted $C\cap D$, and the least upper bound (when it exists) is denoted $C\cup D$. See Appendix A of \cite{Abramenko08} for further discussion of the notion of a simplicial complex.

\begin{Definition} We say that a complex $\Delta$ is a chamber complex if every element is contained in a maximal element and if, given two maximal elements $C, C'$, there exists a finite sequence $C=C_{0}, C_{1}, \ldots C_{m}=C'$, called a gallery of length $m$, such that for all integers $i$ such that $1\leq i \leq m$, the codimension of $C_{i-1}\cap C_{i}$ in either $C_{i-1}$ or $C_{i}$ is at most one. The maximal elements are called chambers. We write $\Cham \Delta$ for the set of chambers of $\Delta$. \end{Definition}

An element of a chamber complex has the same codimension in any chamber which contains it; this quantity is called the codimension of the element of the complex. A chamber complex is called thick (respectively, thin) if every element of codimension one is contained in at least three (respectively, exactly two) chambers. The diameter of a chamber complex is the supremum of the lengths of all the minimal galleries connecting two chambers of the chamber complex.

\bigskip

A flag complex is a complex in which any family of elements any two of which has an upper bound has an upper bound.

\begin{Definition} A morphism of chamber complexes $\Delta,$ $\Delta'$, also called a chamber map, is a mapping $\phi:\Delta\rightarrow \Delta'$ such that the restriction of $\phi$ to the simplex of all faces of any given element $A\in\Delta$ is an isomorphism from the ordered set of all faces of $A$ to the ordered set of all faces of $\phi(A)$, and $\phi$ maps chambers onto chambers. A subcomplex of a chamber complex is a chamber subcomplex if the inclusion mapping is a morphism of chamber complexes. \end{Definition}

\bigskip

An endomorphism of a thin chamber complex $\phi:\Delta\rightarrow\Delta$ is said to be a folding if it is idempotent and every chamber in $\phi(\Delta)$ is the image of exactly two chambers by $\phi$. The range of a folding is called a root. A complex $\Sigma$ is said to be a Coxeter complex if it is a thin chamber complex, and if for every pair $(C,C')$ of adjacent chambers, there exists a folding of $\Sigma$ which maps $C'$ onto $C$. There is a natural one-to-one correspondence between Coxeter complexes and Coxeter groups, whose definition we discuss below.

\bigskip

Suppose that $I$ is a finite set and that $(m_{ij})$ is a matrix indexed by $I\times I$ with entries in $\mathbb{N}\cup\{\infty\}$, such that $m_{ij}\geq 2$ if $i\neq j$, and $m_{ii}=1$ for all $i$. The corresponding Coxeter system is $(W,S)$, where $W$ is a group with generating set $S=\{s_{i}\mid i\in I\}$ and relations $(s_{i}s_{j})^{m_{ij}}=1$ for all $i, j$. For $J\subseteq I$ let $W_{J}$ be the subgroup of $W$ generated by the elements $s_{j}$ for $j\in J$. The mapping $J\mapsto W_{J}$ is an isomorphism of posetes. The left cosets of the groups $W_{J}$ for all $J\subseteq I$ with the reverse of inclusion as the ordering form the Coxeter complex $\Sigma(W,S)$ corresponding to the Coxeter system $(W,S)$. All Coxeter complexes up to isomorphism arise in this way (\cite{Abramenko08}, Chapter 3).

\bigskip

\begin{Definition} If $\Delta$ is a chamber complex and $\mathfrak{A}$ is a family of chamber subcomplexes of $\Delta$ called apartments then we say that $(\Delta,\mathfrak{A})$ is a building if \newline

\begin{enumerate}
\item Any simplex of codimension one meets exactly three chambers (that is, the building is thick); \newline
\item The elements of $\mathfrak{A}$ are Coxeter complexes; \newline
\item Any two elements of $\Delta$ belong to an apartment; \newline
\item If two apartments $\Sigma$ and $\Sigma'$ contain two elements $A, A' \in \Delta$, there exists an isomorphism of $\Sigma$ onto $\Sigma'$ which leaves invariant $A, A'$ and all their faces.
\end{enumerate}

\end{Definition}

It is clear that all the apartments of a building are isomorphic. The isomorphism class of the apartments is called the Coxeter complex of the building. The rank of the Coxeter complex is also called the rank of the building. If the Coxeter complex is isomorphic to the Coxeter system $(W,S)$, then we say that the building has Coxeter system $(W,S)$ or that it is a building of type $(W,S)$. Given that we are working with simplicial complexes of finite rank, it can be shown that a building is always a flag complex (\cite{Abramenko08}, Exercise 4.50). If the Coxeter complex is finite, then the geometric realisation of the Coxeter complex is homeomorphic to a sphere, (\cite{Abramenko08}, Propostion 1.108), and the building is called spherical. It should be noted that the star of any element of a building is again a building. (\cite{Abramenko08}, Proposition 4.9).

\bigskip

If a complex $\Delta$ admits a set of apartments $\mathfrak{A}$ which makes it into a building then the union of all such sets also makes the complex $\Delta$ into a building; hence the complex $\Delta$ has at most one ``maximal building structure" (\cite{Abramenko08}, Section 4.5). If the diameter of $\Delta$ is finite, then the set of apartments $\mathfrak{A}$ is unique when it exists. For these reasons we sometimes by abuse of notation speak of ``the building $\Delta$".

\begin{Definition} Suppose that $G$ is a semisimple algebraic group defined over a field $k$. Then $\Delta(G,k)$ is defined to be the set of all $k$-parabolic subgroups of $G$ with the reverse of inclusion as the order relation. It can be shown that this is a simplicial complex. For each maximal $k$-split torus of $G$ we define the apartment corresponding to this torus to be the set of all $k$-parabolic subgroups containing this torus. The complex $\Delta(G,k)$ with this collection of apartments is a spherical building \cite{Tits74}. This is proved in detail in \S\S20-21 of \cite{Borel91}. \end{Definition}

Given a building $\Delta$ and a chamber $C\in\Delta$ there exists a unique retraction $\lambda_{C}$ of $\Delta$ onto the simplex of all faces of $C$. Two elements of the building are said to be of the same type if their image by $\lambda_{C}$ is the same; this does not depend on the choice of $C$. If we take the quotient of the building by the equivalence relation ``$A$ and $A'$ have the same type" then we obtain the typical simplex of the building typ $\Delta$, and there is a canonical mapping $\mathrm{typ}:\Delta\rightarrow\mathrm{typ}$ $\Delta$. If the building has Coxeter system $(W,S)$ then there is a canonical bijection typ $\Delta\rightarrow\mathfrak{P}(S)$.

\bigskip

Given a Coxeter complex $\Sigma=\Sigma(W,S)$, where $S$ is indexed by $I$, one may define a $W$-invariant double-coset-valued distance function on $\Sigma\times\Sigma$ as follows:

\bigskip

(1) \begin{center}  $\delta:\Sigma\times\Sigma\rightarrow\bigcup\{W_{J}\setminus W/W_{K}, J, K \subseteq I\}$,
$\delta(uW_{J}, vW_{K})=W_{J}u^{-1}vW_{K}$. \end{center}

\bigskip

It can be shown that if $\Delta$ is a building with Coxeter system $(W,S)$, and $\Sigma$ is an apartment of $\Delta$, then the above function $\delta$ on $\Sigma\times\Sigma$ does not depend on the choice of isomorphism of $\Sigma$ with $\Sigma(W,S)$. It can also be shown that if $\Delta$ is a building then there exists a well-defined double-coset-valued distance function

\bigskip

(2) \begin{center} $\delta:\Delta\times\Delta\rightarrow\bigcup\{W_{J}\setminus W/W_{K}, J, K \subseteq I\}$ \end{center}

\bigskip

whose restriction to any apartment of the building is the function $\delta$ given above (\cite{Abramenko08}, Section 4.8).

\bigskip

Next, for convenience, we repeat the definition of topological spherical building taken from \cite{Kramer02}.

\bigskip

Suppose that $C, D$ are chambers in a building $\Delta$ with Coxeter system $(W,S)$, and that $\delta(C,D)=w$, and that $w=s_{i_{1}}s_{i_{2}}\ldots s_{i_{r}}$ is a reduced expression for $w$ in terms of the generating set $S$. Then there exists a unique minimal gallery $C=C_{0}, C_{1}, \ldots C_{r}=D$ such that $\delta(C_{k-1},C_{k})=s_{i_{k}}$ holds for $k=1, 2, \ldots r$. Suppose that $X$ is an element of the bulding $\Delta$ and $C$ is a chamber. Then there exists a unique chamber $E\in$ $\mathrm{St}$ $X$ denoted by $\mathrm{proj}_{X}$ $C$, such that for every chamber $D\in $ $\mathrm{St}$ $X$, and for every minimal gallery $\gamma$ from $C$ to $D$, the first element of $\gamma$ contained in St $X$ is $E$. If $Y\in\Delta$ is arbitrary, there exists a unique $Z$ which is contained in some chamber in St $X$, such that Cham St $Z$=$\mathrm{proj}_{X}$ Cham St $Y$, and we denote this element $Z$ by $\mathrm{proj}_{X} Y$.

\bigskip

\begin{Definition} Suppose that $\Delta$ is a spherical building of type $(W,S)$, considered as a simplicial complex. Suppose that for each $s\in S$ we are given a Hausdorff topology on the set of vertices of type $s$. We thereby obtain, for each $J\subseteq S$, a Hausdorff topology on the simplices of type $J$, by viewing the simplices of type $J$ as $J$-tuples of vertices and taking the subspace topology arising from the product topology. Given any two subsets $J, K \subseteq S$, consider the set $\mathfrak{D}^{J,K}:=\{(X,Y)\mid X$ has type $J,$ $Y$ has type $K,$ there exist opposite chambers containing $X$ and $Y\}$ and the mapping $\mathfrak{D}^{J,K}\rightarrow$ Cham $\Delta, (X,Y)\mapsto\mathrm{proj}_{X} Y$. Suppose that all of these mappings are continuous. Then the building is sad to be a topological spherical building. \end{Definition}

If $k$ is a Hausdorff topological field and $G$ is an semisimple algebraic group defined over $k$, then the building $\Delta(G,k)$ is a spherical building. Let us denote its Coxeter system by $(W,S)$. For each subset $J\subseteq S$, the simplicies of type $J$ form a projective variety over $k$. We can give this variety the strong $k$-topology as defined in Definition \ref{strongtopology}. It is easy to see that this is compatible with the way of deriving the topology on the set of simplicies of a given type of rank greater than one from the topologies on the sets of vertices of each type. We must now show that this makes $\Delta(G,k)$ into a topological spherical building. What we must prove is that the projection maps are continuous. It will be sufficient to show that they are $k$-morphisms of quasiprojective $k$-varieties.

\begin{Definition} Suppose that a group $G$ acts on a building $\Delta$ of type $(W,S)$ in a way that preserves the Weyl distances of chambers and suppose further that $G$ acts transitively on $\{(C,D)\in$ Cham $\Delta\times$ Cham $\Delta\mid\delta(C,D)=w\}$ for each fixed $w\in W$. Then the action of $G$ is said to be Weyl transitive. \end{Definition}

\begin{lemma} Suppose that $G$ is a semisimple algebraic group defined over a Hausdorff topological field $k$. Then the building $\Delta:=\Delta(G,k)$ is a topological spherical building. \end{lemma}

\begin{proof} Suppose that we are given a fixed $w\in W$ with reduced decomposition $s_{1}s_{2}\ldots s_{n}$. Given a pair of chambers $(C,D)$ such that $\delta(C,D)=w$, we may consider the projection of $D$ onto the panel of $C$ with type $s_{1}$, which we denote by $E$. It is sufficient to prove that the mapping defined on $\{(C,D)\mid\delta(C,D)=w\}$ sending each pair $(C,D)$ to the chamber $E$ defined in this way is a $k$-morphism. Given any chamber $C$, let us denote by $G_{C}$ the subgroup of $G$ which fixes $C$. Now, given a pair $(C,D)$ such that $\delta(C,D)=w$, the set $\{(C,D)\mid\delta(C,D)=w\}$ is the set of $k$-rational points of a quasiprojective variety which may be identified with $(G/(G_{C}\cap G_{D}))_{k}$ (see \cite{Humphreys75}, Section 34.2), because the group $G$ acts Weyl transitively on $\Delta$ (see for example \cite{Abramenko08}, Chapter 6). We have $G_{C}\cap G_{D}\subseteq G_{E}$ and the mapping we are considering may be identified with the projection morphism $G/(G_{C}\cap G_{D})\rightarrow G/G_{E}$. Since the groups $G_{C}, G_{D}$ and $G_{E}$ are $k$-subgroups this is clearly a $k$-morphism. \end{proof}

\bigskip

In Chapter 5 of \cite{Tits74}, Jacques Tits classifies the isomorphisms $\Delta(G,k)\rightarrow\Delta(G',k')$, where $k$ and $k'$ are infinite fields and $G$ (resp. $G'$) is a semisimple algebraic group defined over $k$ (resp. $k'$), such that the buildings $\Delta(G,k), \Delta(G',k')$ have Coxeter diagrams with no isolated nodes. Later we shall show how to modify his argument slightly so as to obtain a classification of the injective chamber maps between two of these buildings of the same rank; the only change necessary is that one allows field homomorphisms which are not necessarily surjective rather than field isomorphisms.

\begin{Definition} Suppose that $\alpha$ is a root of an apartment of the building $\Delta$. By the root group $U_{\alpha}$ we mean the group of all type-preserving automorphisms of $\Delta$ which fix every chamber in the root $\alpha$ and also every chamber adjacent to a chamber in $\alpha$ via a panel which is not in the boundary of $\alpha$. \end{Definition}

\begin{Definition} Suppose that a thick spherical building $\Delta$ has the property that for each root $\alpha$ in some apartment of $\Delta$, the root group $U_{\alpha}$ acts simply transitively on the set of apartments containing $\alpha$. Then the building is said to be strictly Moufang (see \cite{Abramenko08}, Chapter 7). \end{Definition}
 
We now review the notion of ``quasi-connected" from Section 1. First we must define a certain base for the topology on the set of chambers of the building of a semisimple algebraic group over a Hausdorff topological field.

\bigskip

Recall that the definition of basic open set can be found in Definition \ref{basicopen}.

\bigskip

We must now show that every basic open set is indeed an open set and that in this way a base for the topology on the set of chambers is defined.

\begin{lemma} \label{basis} A basic open set is an open set, and the basic open sets form a base for the topology on the set of chambers. \end{lemma}

\begin{proof} Let all notations be as in the definition of a basic open set. Suppose that $D\in U$. Then there exist $u_{1}\in N_{s_{1}}, u_{2}\in N_{s_{2}}, \ldots u_{n}\in N_{s_{n}}$ such that $D=u_{1}u_{2} \ldots u_{n}C$. Some of the $u$'s may be equal to the identity. Thus the gallery $(C, u_{1}C, u_{1}u_{2}C, \ldots u_{1}u_{2}\ldots u_{n}C)$ is a gallery joining $C$ to $D$, possibly a stammering gallery. We may find an apartment $A$ containing a non-stammering gallery $(E, v_{1}C, v_{1}v_{2}E, \ldots v_{1}v_{2}\ldots v_{n}E)$ where $v_{i}\in U_{s_{i}}$ for all $i$ such that $1\leq i\leq n$, and further, if $i$ is an integer such that $0\leq i\leq n$, then $u_{1}u_{2}\ldots u_{i}C$ is opposite $v_{1}v_{2}\ldots v_{i}E$. Then the gallery from $C$ to $D$ may be co-ordinatised by the panels of the apartment $A$ via the procedure described in Section 7.2 of \cite{Kramer02}, and the co-ordinatisation map is a homeomorhpism. This shows that there is an open neighbourhood of $D$ contained in $U$, and the second part of the lemma is also now clear. \end{proof}

Recall the definition of ``quasi-connected" given in Definition \ref{quasiconnected}.

\bigskip

Next, we review the definition of a basic open subset of $G^{+}$, where $G^{+}$ is as in the statement of the Borel-Tits theorem.

\begin{Definition} Let $S$ be a maximal $k$-split torus in $G$. Let $\Phi_{k}(G,S)$ be the set of $k$-roots of $G$ with respect to $S$ and let $\Psi=\{\alpha\in\Phi_{k}(G,S)\mid\frac{1}{2}\alpha\notin\Phi_{k}(G,S)\}$. For each $\alpha\in\Psi$ let $U_{\alpha}$ be the subgroup of $G^{+}$ consisting of all $k$-rational points of the unipotent radical of the $k$-parabolic subgroup of $G$ containing $S$ corresponding to $\alpha$, and let $N_{\alpha}$ be an open neighbourhood of the identity in $U_{\alpha}$ in the strong $k$-topology on $U_{\alpha}$. As observed in \cite{Tits65b}, given a system of simple roots $\Pi\subseteq\Psi$ there is a natural one-to-one correspondence between conjugacy classes of $k$-parabolic subgroups and subsets of $\Pi$. Select a system of simple roots $\Pi$ and identify it with the set of generators $S$ for the Coxeter system $(W,S)$ of the building of $G$ over $k$. Pick a longest word $w_{0}\in W$ and let $s_{1}s_{2}\ldots s_{n}$ be a reduced decomposition. For each $i$ such that $1\leq i\leq n$, let $\alpha_{i}$ be the element of $\Pi$ corresponding to $s_{i}$. We say that $N_{-\alpha_{1}}N_{-\alpha_{2}}\ldots N_{-\alpha_{n}}N_{\alpha_{1}}N_{\alpha_{2}}\ldots N_{\alpha_{n}}$ and any right coset thereof is a basic open subset of $G^{+}$. We also say that the empty set is a basic open subset with respect to the torus $S$. \end{Definition}

By considering the simply transitive action of $G^{+}$ on pairs of opposite chambers, it follows from Lemma \ref{basis} that this does indeed define a base for the strong $k$-topology on $G^{+}$.

\bigskip

Our goal in this paper is the proof of Theorems \ref{maintheorem} and \ref{localBorelTits}. We shall give these in the next section.

\section{Proof of the Main Theorems}

Suppose that $G$ (resp. $G'$) is a semisimple algebraic group defined over a field $k$ (resp. $k'$). Suppose that $k$ is a non-discrete Hausdorff topological field. Let $\Delta:=\Delta(G,k), \Delta':=\Delta(G',k')$. Suppose that the Coxeter diagrams of $\Delta, \Delta'$ have no isolated nodes. The building $\Delta$ is a topological spherical building, and we wish to prove that, given a nonempty open quasi-connected subset $U$ of Cham $\Delta$ and an injective chamber map $\varphi:\Delta(U)\rightarrow\Delta'$, there is a unique extension of $\varphi$ to an injective chamber map $\Delta\rightarrow\Delta'$. It is sufficient to prove that if $U\subseteq V$ are basic open sets and $\Delta(U)$ is the domain of the injective chamber map $\varphi$, then there is a unique extension of $\varphi$ to an injective chamber map with domain $\Delta(V)$.

\begin{proof} [Proof of Theorem \ref{maintheorem}] Suppose that $U$ is a basic open set with respect to the pair of opposite chambers $(C, C')$. We denote by $G^{+}$ the subgroup of $G(k)$ generated by all the $k$-rational points of unipotent radicals of $k$-parabolic subgroups, and similarly for $G'^{+}$. First we must show how to associate to the mapping $\varphi$ a local group homomorphism $\psi:W\subseteq G^{+}\rightarrow G'^{+}$, where $W$ is a basic open neighbourhood of the identity in $G^{+}$. 

\bigskip

Let $A$ be an apartment containing $C$ which is wholly contained in $U$. This exists because for each decomposition $w_0=s_{1}s_{2}\ldots s_{k}$ of the longest word $w_{0}\in W$ we may let $U'$ be the set of ends of of a open set $U''$ of non-stammering galleries starting at $C$ of type $(s_{1},s_{2},\ldots s_{k})$ containing the constant gallery, with the property that every member of every gallery in $U''$ is contained in $U$. Taking the intersection of all such $U'$ for every such decomposition, we obtain an open set containing $C$ contained in $U$. A chamber opposite $C$ contained in this set determines together with $C$ an apartment wholly contained in $U$. Let $D$ be the chamber of $A$ opposite $C$. Suppose that the building $\Delta$ is of type $(W,S)$. For each $s\in S$ let $V_{s}$ be the root group corresponding to the root $R_{s}$ of the apartment $A$ which does not contain $C$ but which is attached to $C$ along the panel of cotype $\{s\}$. There exists an open neighbourhood of the identity $W_{s}\subseteq V_{s}$, such that if $g\in W_{s}$, then $g$ maps an open neighbourhood $U_{1}\subseteq U$ of $C$ onto some other open set $U_{2}\subseteq U$, and furthermore $U_{1}$ may be chosen to be a basic open set with respect to the pair of opposite chambers $(C,D)$. From this we obtain a mapping $g^{*}:\varphi(U_{1})\rightarrow\varphi(U_{2})$ given by $g^{*}=\varphi\circ g\circ\varphi^{-1}$. Let $V'_{s}$ be the root group in $G'(k')$ corresponding to the root $\varphi(R_{s})$ of the apartment $\varphi(A)$. 

\begin{lemma} If $g\in W_{s}$ and $U_{1}, U_{2}$ are as above, then $g^{*}$ is the restriction to $\varphi(U_{1})$ of an element of $V'_{s}$. \end{lemma}

\begin{proof} The rigidity theorem, Theorem \ref{rigiditytheorem}, says that an injective chamber map is determined by its action on $E_{1}(D)\cup \{C\}$ where $(C,D)$ are a pair of opposite chambers. The proof of this theorem given in \cite{Abramenko08}, Section 5.9, can be modified to show that if $U_{1}$ is a basic open set with respect to $(C,D)$, and $D'$ is a chamber opposite $C$ such that the apartment generated by $C$ and $D'$ is wholly contained in $U_{1}$, then $g\mid_{U_{1}}$ is determined by its action on $(E_{1}(D')\cap U_{1})\cup \{C\}$. Thus any chamber map $U_{1}\rightarrow U_{2}$ that agrees with $g$ on $(E_{1}(D')\cap U_{1})\cup \{C\}$ must be equal to $g$. It follows that any chamber map $\varphi(U_{1})\rightarrow \varphi(U_{2})$ that agrees with $g^{*}$ on $(E_{1}(\varphi(D'))\cap \varphi(U_{1})\cup \{\varphi(C)\}$ must be equal to $g^{*}$. Since $\Delta'$ is strictly Moufang there is exactly one element of $V'_{s}$ whose restriction to $\varphi(U_{1})$ has this property, so it follows that $g^{*}$ is the restriction to $\varphi(U_{1})$ of an element of $V'_{s}$.

\end{proof}

We have defined a mapping $W_{s}\rightarrow V'_{s}$. Similarly we can define mappings $W_{-s}\rightarrow V'_{-s}$ for each $s\in S$, where $W_{-s}$ is an open neighbourhood of the identity in $V_{-s}$ and $V_{-s}, V'_{-s}$ respectively are the root groups corresponding to $-R_{s}$ in $G(k), G'(k')$ respectively. These mappings extend to a local group homomorphism $\psi:W\subseteq G^{+}\rightarrow G'^{+}$, where $W$ is a basic open neighbourhood of the identity in $G^{+}$ with respect to the maximal $k$-split torus corresponding to the apartment $A$. This group homomorphism also has the property that root groups are mapped into root groups, for any root, not just the simple ones. In fact more can be said: if we let $\Prefix_{k}{\Phi}$ be the set of $k$-roots of $G$ relative to the maximal $k$-split torus corresponding to the apartment $A$, and we let $U_{\alpha}$ be the set of $k$-rational points of the unipotent $k$-group corresponding to $\alpha$ for each $\alpha\in\Prefix_{k}{\Phi}$, and similarly for $U'_{\beta}$ for a $k'$-root $\beta$ of $G'$ with respect to the maximal $k$-split torus corresponding to the apartment $\varphi(A)$, then there is a bijection $f$ between the relative root systems of each group such that $U_{\alpha}\cap W$ is mapped into $U'_{f(\alpha)}$. This follows from the fact that the ``refined RGD-system" structure is determined by the geometry of the building; see Section 7.9.3 of \cite{Abramenko08}, and also Theorem 7.116.

\bigskip

Supose that $\alpha\in\Prefix_{k}{\Phi}, \frac{1}{2}\alpha\notin\Prefix_{k}{\Phi}$. By $U_{(\alpha)}$ we mean the group generated by $U_{k\alpha}$ for all positive integers $k$ such that $k\alpha\in\Prefix_{k}{\Phi}$. Consider the mapping $\psi: W\cap\langle U_{(\alpha)}, U_{(-\alpha)} \rangle\rightarrow \langle U'_{(f(\alpha))}, U'_{(-f(\alpha))} \rangle$. We must try to show that this extends to a global group homomorphism on $\langle U_{(\alpha)}, U_{(-\alpha)} \rangle$. There exists a $k$-split torus $S$ of dimension one, contained in the maximal $k$-split torus corresponding to the apartment $A$, such that $S\cap \alpha^{-1}(k^{*2})\subseteq \langle U_{(\alpha)}, U_{(-\alpha)} \rangle$.  Consider first the case where $2\alpha\notin\Prefix_{k}{\Phi}$. The action of $S$ on $U_{\alpha}$ and $U_{-\alpha}$ by conjugation makes them into vector spaces over $k$. (The group operations on $U_{\alpha}$ and $U_{-\alpha}$ provide the vector space addition in both cases, and the action of $S$ by conjugation gives the scalar multiplication.) Let $W\cap S\cap\alpha^{-1}(k^{*2})$ be equal to $S\cap\alpha^{-1}(W^{*})$ where $W^{*}$ is an open neighbourhood of 1 in $k^{2}$. By considering the action of $S$ on $U_{\alpha}$ and $U_{\alpha}$ we can show that we can extend $\psi$ from $S\cap\alpha^{-1}(W^{*})$ to $S\cap \alpha^{-1}(R)$ where $R$ is the ring generated in $k$ by $W^{*}$. This ring includes an open neighbourhood of 0. We now see that from $\psi$ we can obtain a local field homomorphism $k\rightarrow k'$. (From the action of $\psi$ on $S$ we obtain a mapping defined on $k^{*}$ in a neighbourhood of 1, but this can be transferred to a neighbourhood of 0.)  But local field homomorphisms extend uniquely to global field homomorphisms. It is now easy to see that there is a unique extension of $\psi$ to a global group homomorphism. The argument is similar if $2\alpha\in\Prefix_{k}{\Phi}$, except that we consider the action of $S$ on $U_{(\alpha)}/U_{2\alpha}, U_{2\alpha}, U_{(-\alpha)}/U_{-2\alpha}, U_{-2\alpha}$. Since the restriction of $\psi$ to $W\cap\langle U_{(\alpha)}, U_{(-\alpha)} \rangle$ extends to a unique global group homomorphism on $\langle U_{(\alpha)}, U_{(-\alpha)} \rangle$ for each $\alpha\in\Prefix_{k}{\Phi}$, it follows that $\psi$ extends to a unqiue global group homomorphism on $G^{+}$. From this extension we can now construct an extension of $\varphi$ to an injective chamber map $\varphi:\Delta\rightarrow\Delta'$, or to a unique injective chamber map $\varphi:\Delta(V)\rightarrow\Delta'$ where $V$ is a basic open set containing $U$.

\end{proof}

\begin{proof} [Proof of Theorem \ref{localBorelTits}] We saw above that it is sufficient to prove that the root groups of the refined RGD system are mapped into root groups. The proof of this given in \cite{Tits73} works in the local case with no difficulties. \end{proof}

\section{A corollary dealing with the rigidity of nilpotent groups}

In this section we look at a corollary to the main result of the foregoing section which deals with fibrations of nilpotent Lie groups. Specifically, we prove the following theorem.

\begin{theorem} Suppose that $N$ is a nilpotent real or complex Lie group with Lie algebra $\mathfrak{n}$. Suppose that the Lie algebra $\mathfrak{n}$ occurs in the Iwasawa decomposition of a semisimple Lie algebra $\mathfrak{s}=\mathfrak{k}\oplus\mathfrak{a}\oplus\mathfrak{n}$ whose corresponding Lie group Iwasawa decomposition is $S=KAN$. Suppose that $\mathfrak{n_{1}}$ and $\mathfrak{n_{2}}$ are two linearly disjoint Lie subalgebras of $\mathfrak{n}$ which are each the direct sum of a set of root spaces in $\mathfrak{n}$ and which generate $\mathfrak{n}$ as a Lie algebra. Suppose that for all $X\in\mathfrak{n_{1}}$ there exists a $Y\in\mathfrak{n_{2}}$ such that $[X,Y]\neq 0$, and for all $X\in\mathfrak{n_{2}}$ there exists a $Y\in\mathfrak{n_{1}}$ such that $[X,Y]\neq 0$. Suppose that $N_{1}$ and $N_{2}$ are the analytic groups corresponding to $\mathfrak{n_{1}}$ and $\mathfrak{n_{2}}$. Suppose that $U$ is a nonempty open connected subset of $N$ and that $\phi:U\rightarrow N$ is a mapping whose range is Zariski dense in $N$ such that, for all $x\in U$, $\phi(xN_{1}\cap U)\subseteq \phi(x)N_{1} \cap U$ and $\phi(xN_{2}\cap U)\subseteq \phi(x)N_{2}\cap U$. Then, if we let $k=\mathbb{R}$ or $\mathbb{C}$ depending on whether we are in the real or complex case, $\phi$ is a mapping induced by the left action of an element of $S\cdot \mathbb{C}^{*}1$ on $N=S/KA$, possibly composed with a mapping induced by a field homomorphism $\psi:k\rightarrow k$. \end{theorem}

\begin{proof} Under the hypotheses of the theorem, there must exist two parabolic subgroups of $S$, $P_{1}=KAN_{1}$ and $P_{2}=KAN_{2}$, such that $\phi$ induces a mapping from an open connected set in the family of left cosets of $P_{1}$ into the family of left cosets of $P_{1}$, and similarly with $P_{2}$. Now, by the condition on commutators of elements of $\mathfrak{n_{1}}$ and $\mathfrak{n_{2}}$, the image of any left coset of a parabolic subgroup contained in $P_{1}$ under projection to $S/P_{2}$ can be obtained as the intersection of the projections of a finite family of sets of the form $gP_{1}$ to $S/P_{2}$, and similarly with the roles of $P_{1}$ and $P_{2}$ reversed. We may conclude that $\phi$ induces a mapping from an open connected subset of $S/KA$ to itself whose range is Zariski dense and which preserves the images in $S/KA$ of left cosets of parabolic subgroups. The result now follows from theorem \ref{maintheorem}. \end{proof}

A version of this theorem appears in \cite{Ottazzi11} which requires smoothness. This result can be used to obtain a classification-free and unified version of many of the rigidity results in Yamaguchi \cite{Yamaguchi93}.

\end{document}